\documentclass[11pt]{amsart}

\usepackage{amssymb,amsmath,enumerate,epsfig}
\usepackage{amsfonts,color}
\usepackage{a4,latexsym}
\usepackage[all]{xy}
\usepackage{multirow}

\usepackage{hyperref}
\hypersetup{pdftitle=2lect}

\newcommand{\C} {\mathbb{C}}
\newcommand{\Q} {\mathbb{Q}}

\newcommand{\F}{\mathbb{F}}
\newcommand{\Z}{\mathbb{Z}}
\newcommand{\p}{\mathfrak{p}}

\newcommand{\OO}{\mathcal{O}}

\newcommand{\PP}{\mathbb{P}}
\newcommand{\NS}{\mathop{\rm NS}}

\newcommand{\MW}{\mathop{\rm MW}}

\newcommand{\Km}{\mathop{\rm Km}}
\newcommand{\disc}{\mathop{\rm disc}}

\newcommand{\Br}{\mathop{\rm Br}\nolimits}
\newcommand{\ta}{\tilde A}
\newcommand{\td}{\tilde D}
\newcommand{\te}{\tilde E}
\newcommand{\Aut}{\mathop{\rm Aut}}

\newtheorem{Theorem}{Theorem}
\newtheorem{Proposition}[Theorem]{Proposition}
\newtheorem{Lemma}[Theorem]{Lemma}

\newtheorem{Corollary}[Theorem]{Corollary}

\theoremstyle{remark}
\newtheorem{Remark}[Theorem]{Remark}

\newtheorem{Example}[Theorem]{Example}

\theoremstyle{definition}
\newtheorem{Conjecture}[Theorem]{Conjecture}
\newtheorem{Definition}[Theorem]{Definition}
\newtheorem{Question}[Theorem]{Question}

\begin{document}

\title{Two lectures on the arithmetic of K3 surfaces}

\author{Matthias Sch\"utt}
\address{Institut f\"ur Algebraische Geometrie, Leibniz Universit\"at
  Hannover, Welfengarten 1, 30167 Hannover, Germany}
\email{schuett@math.uni-hannover.de}
\urladdr{http://www.iag.uni-hannover.de/$\sim$schuett/}

%
%
%
%

\date{March 5, 2013}

 \begin{abstract}
In these lecture notes we review different aspects of the arithmetic of K3 surfaces. Topics include rational points, Picard number and Tate conjecture, zeta functions and modularity.
 \end{abstract}
 
 \maketitle
 
 \tableofcontents

\section{Introduction}
\label{s:intro}

K3 surfaces are central objects of study in various areas of mathematics and physics
such as algebraic, complex, and differential geometry, number theory and string theory.
Naturally they featured prominently in the Fields workshop.
These notes record two introductory lectures
on the arithmetic of K3 surfaces
with some bits of additional or supplementary material.
Limitations of space and time do not possibly allow me to do justice
to all important aspects of this area;
I apologise for everything
that may have been left out or not attributed correctly.

In brief these lecture notes aim to shed some light on the following three topics:
\begin{enumerate}
\item
rational points on K3 surfaces;
\item
 Picard numbers and the Tate conjecture for K3 surfaces;
 \item
  zeta functions and modularity for K3 surfaces
  \end{enumerate}
Our survey is initiated by a brief motivation
coming from algebraic curves
which illustrates the thematic interplay between arithmetic and geometry.

\section{Motivation: Rational points on algebraic curves}
\label{s:motiv}
  
Throughout this paper we are mostly concerned with varieties which
are complex, smooth, and projective
although many techniques that we discuss actually involve positive characteristic.
On the level of curves, 
we can equivalently consider compact Riemann surfaces.
Then there is a discrete invariant: the \emph{genus} $g$ of the compact Riemann surface,
i.e.~the number of holes or handles.
It is a non-trivial fact that this purely 
topological invariant
has an algebro-geometric counterpart:
the geometric genus measuring the dimension of the space of regular 1-forms
on the projective curve.
It should come even more surprising how the genus governs the arithmetic of algebraic curves.

To see this, we assume that the algebraic curve $C$ is defined over some number field $K$
(which the reader may just as well assume to be $\Q$).
Then it is a natural problem to investigate the set of $K$-rational points on $C$.
It turns out that the cardinality of $C(K)$
falls into three cases according to the genus of $C$:

$$
\begin{array}{c|c|l}
g(C) & \# C(K) & \text{comment}\\
\hline
0 & 0,\infty & \text{rational if } \; C(K)\neq\emptyset\\
1 & \leq \infty & \text{elliptic if } \; C(K)\neq\emptyset\\
\geq 2 & <\infty & \text{Faltings' Theorem \cite{Fa}}
\end{array}
$$

The genus 1 case is particularly rich since the $K$-rational points form an abelian group
(best visible on the model as a cubic in $\PP^2$ with
the group law that any three collinear points on the curve add up to zero).
It goes back to Mordell and Weil that this group is finitely generated.
Thanks to the group structure on elliptic curves, 
we deduce the following remarkable property of genus 1 curves $C$
over a number field $K$:
there exists some finite extension $K'/K$ 
such that the $K'$-rational points are Zariski dense on $C$.
To see this one first extends $K$ to reach a rational point on $C$
(such that $C$ becomes elliptic)
and then ensures, possibly by a further extension,
that there is a rational point of infinite order.
This concept is usually called potential density:

\begin{Definition}
Let $X$ denote a variety defined over some number field $K$.
We say that $X$ has or satisfies \emph{potential density (of rational points)}
if there exists some finite extension $K'/K$ such that the $K'$-rational points lie dense on $X$.
\end{Definition}

Note that potential density unifies algebraic curves of genus 0 and 1.
By Faltings' Theorem \cite{Fa}, however,
curves of genus greater than 1 \emph{never} satisfy potential density.
It is a common belief that similar structures should hold in higher dimension,
with the genus replaced by the \emph{Kodaira dimension} $\kappa$.
For instance, the Bombieri-Lang conjecture formulates
that on varieties of maximal Kodaira dimension (so called \emph{general type}:
$\kappa(X)=\dim(X)$)
the Zariski closure of the rational points over any number field forms a proper subvariety.
It seems worthwhile noticing that
potential density makes the arithmetic problem of rational points
into a geometric notion which only depends on the $\bar \Q$-isomorphism class
of the variety, but not on the precise model chosen.

In the next section, we discuss the problem of rational points on K3 surfaces
which can be considered as a 2-dimensional analogue of elliptic curves.
Later we will see that also other concepts such as modularity carry over
from elliptic curves to certain K3 surfaces in a decisive way.

\section{K3 surfaces and rational points}
\label{s:rat}

In essence there are two ways to extend the definition of elliptic curves to dimension 2.
Requiring a group structure leads to abelian varieties
which are fairly well understood in arithmetic and geometry.
Almost automatically they come with potential density.
On the other hand, we can impose the Calabi-Yau condition (in the strict sense);
this leads to the notion of K3 surfaces:

\begin{Definition}
A smooth projective surface $X$ is called K3 if
\[
\omega_X \cong \OO_X \;\;\; \text{ and }\;\;\; h^1(X,\OO_X)=0.
\]
\end{Definition}
In fact, all K3 surfaces can be seen to be (algebraically) simply connected, and over $\C$,
deformation equivalent (although the original argument for this went through non-algebraic K3 surfaces).
For complex K3 surfaces we record the Hodge diamond which can be computed 
easily with Noether's formula:
$$
\begin{array}{ccccc}
&& 1 &&\\
& 0 && 0 &\\
1 && 20 && 1\\
& 0 && 0 &\\
&& 1 &&
\end{array}
$$
The resulting Betti numbers also hold in positive characteristic (for $\ell$-adic \'etale cohomology, say).

We give three examples.
The first two mimic the definition of elliptic curves in 2 essentially different ways 
(which will surface again in \S\ref{s:rho=1})
while the third relates to abelian surfaces.
To ease the exposition, we limit ourselves to constructions outside characteristic 2.

\begin{Example}
\label{ex}
\begin{enumerate}
\item
Smooth quartics in $\PP^3$.
\item
Double sextics,
i.e.~double coverings $X\to\PP^2$
branched along a smooth sextic curve.
\item
Kummer surfaces $\Km(A)$
where $A$ is an abelian surface
and $\Km(A)$ is the minimal resolution of the quotient $A/\langle-1\rangle$
with 16 rational double points.
\end{enumerate}
\end{Example}

In view of the deformation equivalence,
we can also allow the quartic or sextic above to have isolated rational double points
as singularities (ADE-type)
and consider a minimal resolution which will then be K3.

In the following we will repeatedly consider Kummer surfaces of product type
where the abelian surface is isomorphic to a product of elliptic curves $E\times E'$.
Such Kummer surfaces come naturally with models as quartics or double sextics.
To see this in an elementary way, write the elliptic curves in extended Weierstrass form
\begin{eqnarray}
\label{eq:E,E'}
E: \;\; y^2 = f(x),\;\;\; E': y'^2 = f'(x')
\end{eqnarray}
with cubic polynomials $f, f'$ without multiple roots.
Then a birational model of the Kummer surface $\Km(E\times E')$ is given by
the double sextic
\begin{eqnarray}
\label{eq:Km}
\Km(E\times E'): \;\; w^2 = f(x) f'(x').
\end{eqnarray}
Similarly quartic models are derived by bringing two linear factors from the RHS (over $\bar K$)
to the LHS (multiply $w$ by these factors).
The fact that all these constructions produce indeed isomorphic K3 surfaces
relies on general surface theory (birational maps between K3 surfaces are isomorphisms).
There is one more incarnation of K3 surfaces that comes up handily on Kummer surfaces of product type:
elliptic fibrations.
Before sketching their theory in the next section,
we indicate the relevance to the question of rational points:

\begin{Theorem}[Bogomolov, Tschinkel {\cite{BT}}]\label{Thm:BT}
Let $X$ be a K3 surface over a number field. 
If $X$ has an elliptic fibration or infinite automorphism group,
then $X$ satisfies potential density.
\end{Theorem}

The case of infinite automorphism group naturally implies potential density;
indeed it suffices to exhibit a rational curve on $X$ whose orbit under $\Aut(X)$ is infinite.
Note that automorphisms of K3 surfaces are defined over number fields
since the automorphism group is discrete and finitely generated (see \cite[\S7 Thm.~1]{PSS}, \cite{Sterk} and also \cite[Prop.~2.1]{HSa}).
We will briefly explain the idea behind the elliptic fibrations 
after introducing the necessary background in the next section.

Apropos automorphisms, we mention as a sample of another yet completely different set of problems the question
of the distribution of rational points, and in particular their periodicity under automorphisms.
These issues lend K3 surfaces to the subject of dynamics.
For instance, in \cite{Silverman} it is proved for certain K3 surfaces with infinite automorphism group
(intersections of hypersurfaces of bidegree $(1,1)$ and $(2,2)$ in $\PP^2\times\PP^2$)
that the orbit of a rational point under the automorphism group is either finite or Zariski dense.
The key ingredient here is a new notion of canonical height.

\section{Elliptic K3 surfaces}
\label{s:ell}

An elliptic surface is a smooth projective surface $X$ together with a surjective morphism to a projective curve $C$,
\[
X \to C,
\]
such that almost all fibers are smooth curves of genus 1.
Often one assumes the existence of a section (so that all fibers are in fact elliptic curves over the base field), 
but we will not restrict to these so-called \emph{jacobian} fibrations here.
In order to rule out products, for instance,
one usually assumes that the fibration has a singular fiber.
The possible singular fibers have been classified by Kodaira over $\C$ \cite{K};
later Tate exhibited an algorithm for any perfect base field \cite{Tate}.
The reducible singular fibers consist solely of $(-2)$-curves (smooth rational curves)
whose configuration corresponds to an extended Dynkin diagram (types $\ta_n, \td_k, \te_l$).
The only irreducible singular fibers are the nodal and the cuspidal cubic.

\begin{Example}[Kummer surface of product type]
The projections from $E\times E'$ to either factor induce elliptic fibrations on $\Km(E\times E')$.
In the notation of \eqref{eq:E,E'}, they could be given by twisted Weierstrass forms
\begin{eqnarray}
\label{eq:Km2}
f(x) y^2 = f'(x')
\end{eqnarray}
where $x$ represents an affine parameter of the base curve $\PP^1$.
Visibly the fibration is isotrivial, with all smooth fibers quadratic twists of $E'$ (in particular $\bar K$-isomorphic).
The singular fibers (all of them reducible) are located at $\infty$ and the roots of $f(x)$,
i.e.~they correspond to the 2-torsion points of $E$.
Each singular fiber has Kodaira type $I_0^*$ (corresponding to $\td_4$),
a double $\PP^1$ blown up in 4 $A_1$ singularities.
\end{Example}

One advantage of elliptic surfaces is that they allow us to control  the N\'eron-Severi group $\NS$
to some extent.
Notably any two fibers are algebraically equivalent,
but components of reducible fibers contribute non-trivially to $\NS$.
We call these curves \emph{vertical} --
as opposed to the multisections which are imagined in the horizontal direction.
Clearly $\NS$ of an elliptic surface is generated by horizontal and vertical divisors;
it is a non-trivial fact, however, that once there is a section, one can generate $\NS$ exclusively
by fiber components and sections (see \cite{ShMW}).
In particular, there is a closed expression (often referred to as the Shioda-Tate formula, 
cf.~\cite[Cor.~5.3]{ShMW}) for the Picard number
\[
\rho(X)=\mbox{rank}\NS(X),
\]
involving only the reducible fibers (more precisely the number $m_v$ of components of the fiber $F_v$)
and the Mordell-Weil rank $r$:
\begin{eqnarray}
\label{eq:ST}
\rho(X) = 2 + r + \sum_{v\in C} (m_v-1).
\end{eqnarray}
Since the Picard number of an elliptic surface and its jacobian are the same,
this formula indirectly also applies to any elliptic surface without section.
Throughout the paper the Picard number should always be understood geometrically,
i.e.~over the algebraic closure of the base field
(although we consider varieties over non-closed fields).

As an illustration, potential density holds for any jacobian elliptic (K3) surface with positive Mordell-Weil rank.
To prove Theorem \ref{Thm:BT},
one is thus led to consider elliptic K3 surface with $\MW$-rank zero or no sections at all.
In brief 
Bogomolov and Tschinkel show that any elliptic K3 surface with Picard number $\rho\leq 19$ admits infinitely many suitable multisections
which are rational.
Then they continue to prove that enough of these multisections are not related to torsion points.
To finish the proof, they refer to a result by Inose and Shioda \cite{SI}
that any K3 surface with $\rho=20$ has infinite automorphism group
(derived in the framework of Shioda-Inose structures, see \S\ref{s:SI}).

We now turn to the problem
how restrictive the assumptions in Theorem \ref{Thm:BT} are.
There are at least two answers:
\[
\text{fairly restrictive \;\;\; or \;\;\; not terribly restrictive}.
\]
To justify the second answer, we mention a special feature of K3 surfaces:
elliptic fibrations are completely governed by lattice theory.
Namely, any divisor of self-intersection zero induces an elliptic fibration 
after \cite[\S3, Thm.~1]{PSS}.
Here one first applies reflections and inversions to $D$ until it becomes effective by Riemann-Roch.
Upon subtracting the base locus, the resulting linear system  induces the elliptic fibration.
One easily deduces:

\begin{Lemma}
\label{Cor:ell-K3}
Any K3 surface with Picard number $\rho\geq 5$ admits an elliptic fibration.
\end{Lemma}

\begin{proof}
The intersection pairing between curves
equips $\NS$ with a non-degenerate quadratic form of signature $(1,\rho-1)$
(compatible with cup-product on $H^2(X,\Z)$).
Since any such quadratic form of rank at least 5 represents zero,
the claim follows from the discussion preceding the lemma.
\end{proof}

Thus we find that the assumptions of Theorem \ref{Thm:BT} are not terribly restrictive
in the following sense:

\begin{Corollary}
Any K3 surface of Picard number $\rho\geq 5$ over a number field
satisfies potential density.
\end{Corollary}

In the opposite direction,
it has to be noted that either assumption of Theorem \ref{Thm:BT} implies $\rho>1$.
A generic K3 surface, however, has $\rho=1$.
To argue that the assumptions are \emph{fairly restrictive},
it therefore suffices to rule out that K3 surfaces over number fields
somehow happen to lie on the countably many hypersurfaces in the moduli spaces
comprising K3 surfaces with $\rho>1$.
This question will be discussed both from the theoretical and explicit view point
in the next section.

\section{Picard number one}
\label{s:rho=1}

An elliptic curve always possesses a model as a plane cubic
thanks to Riemann-Roch.
Quite opposite to this, K3 surfaces have many different incarnations.
We have seen two of them in Example \ref{ex}:
double sextics on the one hand and quartics in $\PP^3$ on the other.
While these cases certainly overlap, for instance on the Kummer surfaces (not only of product type),
they differ in an essential way.
This can be seen as follows.

By general moduli theory, both generic double sextic and  generic quartic have Picard number $\rho=1$
(for quartics, this result originally goes back to a conjecture of Noether, as proved by Tjurina).
But then the hyperplane section $H$ gives an ample divisor of self-intersection
\[
H^2 = 2 \;\; \text{ resp. }\;\; H^2 = 4.
\]
In other words, $\NS$ as a lattice equals $\Z\langle 2\rangle$ reps.~$\Z\langle 4\rangle$.
As these two lattices are not isometric, generic double sextics and quartics cannot be isomorphic.
In fact for any integer $d>0$,
there are K3 surfaces with a so-called polarization of degree $2d$
forming a 19-dimensional moduli space.
The uniform approach would be to consider these as hypersurfaces
in the 20-dimensional moduli space of all K3 surfaces (including non-algebraic ones).
Similarly, K3 surfaces of Picard number $\rho\geq 2$ (such as in Theorem \ref{Thm:BT})
lie on hypersurfaces in the moduli spaces of polarized K3 surfaces,
described by lattice polarisations (cf.~\cite{Mo}).
The solution whether all K3 surfaces over number fields might somehow happen
to lie on these hypersurfaces was given by Terasoma and Ellenberg:

\begin{Theorem}[Terasoma {\cite{Terasoma}}, Ellenberg {\cite{Ell}}]
For any integer $d>0$,
there exist $2d$-polarised K3 surfaces over $\bar\Q$ with $\rho=1$.
\end{Theorem}

In this sense, the assumptions of Theorem \ref{Thm:BT} have to be considered as fairly restrictive.
In particular, there was no K3 surface of $\rho=1$ known to satisfy potential density
until very recently Kharzemanov announced in \cite{Ka} the existence of such  K3 surfaces.
%
In the sequel we shall discuss the first big obstacle to producing such examples:
it is very hard to exhibit explicit K3 surfaces with $\rho=1$!

\section{Computation of Picard numbers}
\label{s:rho}

The crux with the Picard number of an algebraic surface $X$ is that it is in general very hard to compute.
That is, unless the geometric genus of $X$ vanishes -- over $\C$ or if the surface lifts to 
such a surface in characteristic zero.
In that case, we have $H^{1,1}(X) = H^2(X,\C)$
so that Lefschetz' theorem returns
\begin{eqnarray}
\label{eq:Lef}
\NS(X) = H^{1,1}(X) \cap H^2(X,\Z) = H^2(X,\Z).
\end{eqnarray}
Here we trivially deduce $\rho(X)=b_2(X)$.
In case of non-zero geometric genus (or non-liftability), 
we would only be aware of the following procedure to compute the Picard number:
\begin{enumerate}
\item[1.]
in the daytime, search systematically for (independent) curves on $X$
giving lower bounds for $\rho(X)$;
\item[2.]
in the nighttime, develop upper bounds for $\rho(X)$ until upper and lower bounds match.
\end{enumerate}

We have to remark that it is unclear as of today whether either step can be implemented
in an effective way.
For instance, the daytime step requires to check on $X$ for curves of increasing degrees
which soon becomes computationally fairly expensive.
There are, however, situations where we can do better.
Notably the daytime  step becomes vacuous if we aim to prove $\rho(X)=1$.
Also if $X$ admits a jacobian elliptic fibration, then the shape of the curves to consider is very clear, as 
they can all be taken as sections where one raises the height successively (see the discussion at the end of Section
\ref{s:mod} where this plays a crucial role).

The nighttime step concerning upper bounds for $\rho$ is yet more delicate.
A priori, one has only the following estimates 
\[
\rho(X) \leq 
\begin{cases}
h^{1,1}(X) & \text{(over $\C$ by Lefschetz)}\\
b_2(X) & \text{(in any characteristic due to Igusa)}.
\end{cases}
\]
Currently there are two approaches that have been worked out and tested in detail.
The first requires the special situation where $X$ admits non-trivial {\bf automorphisms}.
These give extra information about algebraic and non-algebraic classes in $H^2(X)$.
The extreme situation consists of Fermat varieties where the big automorphism group
gives complete control over all cohomology groups (see Example \ref{ex:Fermat}). 
More generally, as soon as an automorphism
acts non-trivially on the regular 2-forms on the surface $X$, 
this can give upper bounds for $\rho(X)$. 
In his pioneering work \cite{Sh-PicV}
Shioda used this technique to exhibit an explicit quintic surface over $\Q$ with $\rho=1$
(quite surprisingly one might want to add).

The second approach towards upper bounds for $\rho$
consists in smooth {\bf specialization}, mostly to characteristic $p>0$.
Here we consider a complex surface $X$ with a model over some number field $K$ 
(or its $\p$-adic completion) with good reduction $X_\p$ at some prime $\p$.
Since intersection numbers are preserved under specialization,
we obtain an embedding of lattices
\begin{eqnarray}
\label{eq:spec}
\NS(X) \hookrightarrow \NS(X_\p).
\end{eqnarray}
Directly this gives the upper bound
\[
\rho(X) \leq \rho(X_\p).
\]
Here the big advantage is that this upper bound is (theoretically) explicitly computable
assuming the Tate conjecture (see Conj.~\ref{conj}).
Concretely $X_\p$ is equipped with the Frobenius automorphism Frob$_\p$
raising coordinates to their $q$-th powers
where $q=\#\F_\p=p^r$ is the norm of $\p$.
Then one is led to consider the induced action of Frob$_\p^*$ on $H^2_\text{\'et}(X_\p,\Q_\ell)$.
A crucial property of $\NS$ in this context is that it can always be generated by divisors
defined over some finite extension of the base field.
This implies that the absolute Galois group acts on $\NS$ through a finite group.
Embedding 
\begin{eqnarray}
\label{eq:NS}
\NS(X_\p)\hookrightarrow H^2_\text{\'et}(X_\p,\Q_\ell)
\end{eqnarray}
via the cycle class map,
we find that all eigenvalues of Frob$_\p^*$ on the image of $\NS(X_\p)$ take the shape $\zeta q$
where $\zeta$ runs through roots of unity.

\begin{Conjecture}[Tate {\cite{Tate-C}}]
\label{conj}
All eigenspaces of Frob$_\p^*$ in $H^2_\text{\'et}(X_\p,\Q_\ell)$ with eigenvalues as above
are algebraic.
\end{Conjecture}

The Tate conjecture has been known, for instance, for Fermat varieties
and several kinds of K3 surfaces including elliptic ones \cite{ASD} and those of finite height \cite{NO}.
Recently, intriguing finiteness statements have been discovered to be equivalent to the Tate conjecture in \cite{LMS}.\footnote{Added in proof:  After this paper was written,
Charles \cite{Charles-Tate} and Madapusi Pera \cite{Madapusi} have announced independent
proofs of the Tate Conjecture for K3 surfaces outside characteristic $2$ (and $3$ in Charles' case).} 
At any rate, the above discussion gives an upper bound for $\rho(X_\p)$ in terms of the eigenvalues of Frob$_\p^*$ on $H^2_\text{\'et}(X_\p,\Q_\ell)$.
We shall now indicate how to compute these eigenvalues.

The reciprocal characteristic polynomial $P_2(T)$ of Frob$_\p^*$ on $H^2_\text{\'et}(X_\p,\Q_\ell)$
appears as a factor of the zeta function of $X_\p$ over $\F_\p$.
By the Weil conjecture, the zeta function can be computed by point counting
over sufficiently many finite fields $\F_{q^r}$ through Lefschetz' fixed point formula.
For instance, if $X_\p$ is a regular surface over $\F_q$ (in the sense that $b_1(X_\p)=0$),
then its zeta function takes the shape
\begin{eqnarray}
\label{eq:zeta}
\zeta(X,T) = \dfrac 1{(1-T)~ P_2(T)~ (1-q^2T)}
\end{eqnarray}
Hence
point counting up to $r=\lceil(b_2(X_\p)-1)/2\rceil$ will be sufficient 
to compute the zeta function thanks to the functional equation.
Note that this will in practice still be fairly expensive,
but there are improvements using $p$-adic cohomology \cite{AKR}.

With all these techniques at hand,
here comes the major drawback of the specialization method:
assuming the Tate conjecture,
non-algebraic eigenclasses of Frob$_\p^*$ in $H^2_\text{\'et}(X_\p,\Q_\ell)$ come in pairs,
corresponding to pairs of complex-conjugate eigenvalues which are not multiples of roots of unity by $q$
(but algebraic integers of absolute value $q$ by the Weil conjectures).
In particular this would imply 
\[
\rho(X_\p) \equiv b_2(X_\p) \mod 2.
\]
Thus,
if we want to prove that some surface $X$ has Picard number $\rho(X)$ of parity other than 
that of $b_2(X)$, one has to be more inventive.
In the next section we sketch what has been done (and what might be done)
for the prototype case of K3 surfaces with $\rho=1$.

\section{K3 surfaces of Picard number one}
\label{s:K3-1}

This section describes how to attack the computation of the Picard number of a K3 surface
over some number field.
We explain in some detail the prototype case of $\rho=1$.
The first one to exhibit an explicit K3 surface $X$ with $\rho(X)=1$ was van Luijk.

\subsection{Van Luijk's approach}
\label{ss:vL}

In  \cite{vL} van Luijk exhibited a K3 surface as  quartic $X$ over $\Q$
with two different primes $\p, \p' (=2,3)$ of good reduction
where the point counting method from the previous section gave the upper bound
\begin{eqnarray}
\label{eq:2}
\rho(X) \leq \rho(X_\p),\rho(X_{\p'}) \leq 2.
\end{eqnarray}
Assuming that $\rho(X)=2$, the embeddings of lattices
\begin{eqnarray}
\label{eq:emb}
\NS(X_{\p'}) \hookleftarrow \NS(X) \hookrightarrow \NS(X_\p)
\end{eqnarray}
would be of finite index. In particular, this would imply that the discriminants of all three lattices
(i.e.~the determinants of the Gram matrices for a basis)
would be the same up to some square factors.
This property, however, can lead to a contradiction by working out explicit basis
for both $\NS(X_\p), \NS(X_{\p'})$ 
and verifying that the intersection forms  are not compatible.

\subsection{Kloosterman's improvement}
\label{ss:Kl}

Subsequently Kloosterman noticed that it is possible to circumvent the determination of generators 
of $\NS(X_\p)$ and $\NS(X_{\p'})$.
In a similar situation in \cite{Kl-15} he instead appealed to the Artin-Tate conjecture \cite{AT}
which while equivalent to the Tate conjecture by \cite{Milne},
additionally predicts the square class of the discriminant of $\NS$:
\begin{eqnarray}
\label{eq:ATC}
\disc \NS(X_\p) \stackrel{?}{=} q \dfrac{P_2(T)}{(1-qT)^{\rho(X_\p)}}\Big |_{T=1/q} \in \Q^*/(\Q^*)^2.
\end{eqnarray}
Note that this procedure does not actually require the validity of the Tate conjecture.
For if the Tate conjecture were to be wrong for $X_\p$ in the above setting,
that is $\rho(X_\p)<2$,
then automatically $\rho(X)=\rho(X_\p)=1$ by \eqref{eq:2}
which is exactly the original claim.
On the other hand, if the Tate conjecture is valid for $X_\p$ and $X_{\p'}$,
then we read off the square classes of $\NS(X_\p)$ and $\NS(X_{\p'})$ from \eqref{eq:ATC}.
If they do not agree, then we derive the desired contradiction to the assumption $\rho(X)=2$.

\subsection{Elsenhans--Jahnel's work}

The method pioneered by van Luijk takes a substantial amount of computation time and memory
since point counting over fairly large finite fields is required 
for two suitable primes of good reduction (with no guarantee that $2$ and $3$ would work).
With a view towards double sextics (which often have bad reduction at $2$),
Elsenhans and Jahnel modified van Luijk's method in such a way
that point counting is only required at one suitable prime.
Based on work by Raynaud \cite{Ray}
 they show in \cite{EJ} that the embedding \eqref{eq:NS}
is primitive (i.e.~the cokernel is torsion-free)
 in a wide range of cases including smooth surfaces over $\Q$ with good reduction at $\p\neq 2\Z$.
(In consequence the conjectural finite index embeddings in \eqref{eq:emb} would in fact be isometries of lattices.)
In practice, this means the following:

\begin{Proposition}[Elsenhans-Jahnel]
In the above setup, assume $\p\neq 2\Z$.
If some divisor class does not lift from $X_\p$ to $X$,
then $\rho(X)<\rho(X_\p)$.
\end{Proposition}

In order to exhibit a K3 surface over some number field with $\rho=1$,
it thus suffices to find a single prime $\p$ of good reduction such that
\begin{enumerate}
\item
$\rho(X_\p)\leq 2$ by inspection of the characteristic polynomial of Frob$_\p^*$ on 
$H^2_\text{\'et}(X_\p,\Q_\ell)$ and
\item
some divisor class on $X_\p$ does not lift to $X$.
\end{enumerate}

\subsection{Outlook}

Currently van Luijk and the author are working on an arithmetic deformation technique
that would allow to construct explicit K3 surfaces with $\rho=1$ 
(and other prescribed Picard numbers)
without any point counting at all.
The overall idea is to combine the above techniques
with extra information which can be extracted from automorphisms
(see \S\ref{s:rho}).
While we will actually mainly aim at Picard numbers of quintics and beyond,
this method can be used, for instance, 
to prove that the following double sextic has $\rho(X)=1$ over $\C$:
\[
X:\;\;\; w^2 = x^5+xy^5+101y^4+1.
\]

\subsection{Feasibility}

Having exhibited explicit K3 surfaces with $\rho=1$,
we shall now come to the problem
whether the above techniques may be applied to all K3 surfaces over number fields.
That is, we ask for an algorithm to compute the Picard number of a K3 surface 
which always terminates theoretically.
This issue was taken up in a recent preprint by Charles \cite{Charles}.

In detail, it is shown using the endomorphism algebra $E$ of the Hodge structure
underlying the transcendental lattice (cf.~\eqref{eq:T})
that one cannot in general expect that there are  primes $\p$
such that $\rho(X_\p)\leq \rho(X)+1$.
However, Charles proves that there are always infinitely many primes $\p$ such that 
\[
\rho(X_\p) = \rho(X) \;\; \text{ or } \;\; \rho(X_\p) = \rho(X) + \dim_\Q E.
\]
Additionally, in the latter case, there are different discriminants turning up on the reductions.
Once one knows $E$, these facts facilitate an algorithm 
which theoretically returns the Picard number of $X$.
However, the determination of $E$ (by similar methods) seems to require the validity of
the Hodge conjecture for the self-product $X\times X$.

\section{Hasse principle for K3 surfaces}
\label{s:Hasse}

Coming back to rational points on K3 surfaces,
there are many more subtle problems to investigate.
Here we comment on recent developments concerning the Hasse principle
which in fact relate to the Picard number one problem as well.

Given a variety $X$ over a number field $K$,
one may wonder whether for the existence of a global point on $X$
 it suffices to have local points over $K_v$ for every place $v$ of $K$.
 This is called the {\bf Hasse principle},
 phrased in terms of the ad\`eles $\mathbb{A}_K$:
 \begin{eqnarray}
 \label{eq:Hasse}
 X(\mathbb{A}_K)\neq\emptyset \stackrel{?}{\Longrightarrow} X(K)\neq\emptyset.
 \end{eqnarray}
 The classical case for the Hasse principle to hold consists in conics in $\PP^2$,
 but already for cubics in $\PP^2$ it may fail by an example due to Selmer \cite{Selmer}.
 Often this failure can be explained by the Brauer group
 \[
 \Br(X) = H_\text{\'et}^2(X,\mathbb{G}_m).
 \]
In fact, via local invariants any subset $S\subseteq \Br(X)$ gives rise to an intermediate set
\[
X(K) \subseteq X(\mathbb{A}_K)^S \subseteq X(\mathbb{A}_K)
\]
which may be empty even if $X(\mathbb{A}_K)$ is not (see \cite[\S 5.2]{Sko}).
This is exactly the situation of a Brauer-Manin obstruction to the Hasse principle for $X$
as pioneered by Manin in \cite{Manin}.
In practice one specifies two subgroups
\[
\Br_0(X) \subseteq \Br_1(X) \subseteq \Br(X)
\]
as follows:
\begin{eqnarray*}
\text{constant} &&
\Br_0(X) = \mbox{im}(\Br(k) \to \Br(X))\\
\text{algebraic} && 
\Br_1(X) = \ker(\Br(X) \to \Br(X \otimes \bar K)
\end{eqnarray*}
Class field theory shows for any $S\subseteq\Br_0(X)$ 
that $X(\mathbb{A}_K)^S =X(\mathbb{A}_K)$.
Algebraic Brauer-Manin obstructions to the Hasse principle 
(and to the related concept of weak approximation, i.e.~density of $X(K)$ in the product of all $X(K_v)$)
have been studied extensively in the last 40 years (see the references in \cite{HV}).
Meanwhile it is the {\bf transcendental} elements in $\Br(X)\setminus\Br_1(X)$
that have resisted concrete realizations;
in fact, most constructions in the last decade have only impacted weak approximation
while relying in an essential way on elliptic fibrations (so that $\rho\geq 2$).
This was rectified recently by a remarkable construction due to Hassett and V\`arilly-Alvarado \cite{HV}:
they exhibit K3 surfaces with Picard number one (as double sextics over number fields)
with explicit quaternion algebras in $\Br(K(X))$ 
giving rise to a transcendental Brauer-Manin obstruction to the Hasse principle.
Their work builds on results of van Geemen on Brauer groups of K3 surfaces \cite{vG}
and extends previous results which only applied to weak approximation \cite{HVV}.

\section{Rational curves on K3 surfaces}

To close our considerations about rational points of K3 surfaces,
we briefly comment on the related topic of rational curves.
The fundamental problem is:

\begin{Question}
Does any K3 surface contain infinitely many rational curves?
\end{Question}

We have already touched upon an answer for K3 surfaces with $\rho\geq 5$:
these admit an elliptic fibration (Cor.~\ref{Cor:ell-K3})
with infinitely many rational multisections (see the discussion in Section \ref{s:ell}).
The problem has seen amazing progress recently, starting from 
\cite{BHT} and greatly extended in \cite{LL}.
The main idea is to first reduce to K3 surfaces $X$ over $\bar\Q$ and then use reduction mod $p$.
Then the odd parity of $\rho(X)$ (and the Tate conjecture) implies
the existence of additional curves on $X_\p$ -- including infinitely many rational ones
as one can show.
Then one lifts back based on arguments going back to Bogomolov and Mumford.

\begin{Theorem}[Bogomolov-Hassett-Tschinkel {\cite{BHT}}, Li-Liedtke {\cite{LL}}]
Any K3 surface over $\bar\Q$ with odd $\rho$ or $\rho\geq 5$  contains infinitely many rational curves (over $\bar\Q$).
\end{Theorem}

We remark that the theorem as it stands does not imply potential density 
because there is no control over the fields of definition of the rational curves.

\section{Isogeny notion for K3 surfaces}
\label{s:Inose}

Having said that Picard numbers are hard to compute,
there is a big advantage when working with complex K3 surfaces.
This files under a notion of isogeny introduced by Inose:

\begin{Proposition}
\label{Prop:Inose}
Let $X, X'$ be complex K3 surfaces
admitting a dominant rational map $X\dasharrow X'$.
Then $\rho(X)=\rho(X')$.
\end{Proposition}

The proof is an easy exercise using Hodge structures. 
Essentially one only has to consider the blow-up $\tilde X$ of $X$ along the locus
of indeterminacy of the rational map:
$$
\begin{array}{ccc}
\tilde X &&\\
\downarrow & \searrow &\\
X & \dasharrow & X'
\end{array}
$$
Then we can pull-back the transcendental  Hodge structures from $X$ and $X'$ to $\tilde X$.
Since the geometric genus is always 1,
all these Hodge structures are determined as the smallest $\Q$-sub-Hodge structure of 
$H^2(\tilde X,\Q)$
whose complexification contains $H^{2,0}(\tilde X)$.
In particular, they are isomorphic as  $\Q$-Hodge structures.
In consequence the Picard numbers of $X$ and $X'$ coincide.
\qed

\begin{Remark}
It was pointed out  by Shioda
that the analogue of Proposition \ref{Prop:Inose} in positive characteristic fails in general,
as there are unirational supersingular K3 surfaces (e.g.~Kummer surfaces).
There seems to be a proof, though, for K3 surfaces of finite height
(such that the Tate conjecture holds true by \cite{NO}).
\end{Remark}

As an application, we can consider symplectic automorphisms of K3 surfaces,
i.e.~those leaving the regular 2-form invariant.
Over $\C$, Nikulin proved in \cite{Nik}
that the fixed locus always consists of a certain finite number of isolated fixed points
which only depends on the order of the automorphism (at most 8).
On the quotient surface, the fixed points  yield rational double point singularities
whose resolution thus is again a K3 surface.
By Proposition \ref{Prop:Inose} the Picard numbers are the same.

\begin{Example}
Let $X$ be a K3 surface admitting a jacobian elliptic fibration with a torsion section.
Then translation by the section defines an automorphism of the underlying surface.
The quotient surface $X'$ is naturally endowed with an elliptic fibration
such that the quotient map of the K3 surfaces corresponds to an isogeny of the generic fibers 
as elliptic curves over the function field of $\PP^1$.
Here, of course, there is a dual isogeny $X'\dasharrow X$.
(This does not hold in general for symplectic involutions of K3 surfaces.)
\end{Example}

\section{Singular K3 surfaces}
\label{s:sing}

The arithmetic of K3 surfaces is conceivably best understood for big Picard number.
In the workshop this could be witnessed in several talk,
for instance by Bertin, Clingher, Elkies, Kumar and Whitcher.
To begin with, we shall concentrate on the case of maximal Picard number over $\C$
in view of Lefschetz' bound in \eqref{eq:Lef}:

\begin{Definition}
A complex K3 surface $X$ is called \emph{singular}
if $\rho(X)=20$.
\end{Definition}

Note that singular K3 surfaces are smooth by definition;
the phrase "singular" is used in the sense of "exceptional".
In fact, there is an analogy with  elliptic curves
with complex multiplication (CM)
which will become clear in \S\ref{s:SI} (see also Remark \ref{rem:CM}).
In the sequel, the Fermat quartic will serve as our guiding example:

\begin{Example}[Fermat quartic]
\label{ex:Fermat}
The Fermat quartic surface
\[
S = \{x_0^4+x_1^4+x_2^4+x_3^4=0\}\subset\PP^3
\]
has $\rho(S)=20$; thus it defines a singular K3 surface.
There are several ways to see this.
Intrinsically one could appeal to the general theory of Fermat varieties.
These come with a big automorphism group 
whose eigenspaces in cohomology can be described in combinatorial terms
in such a way that one can read off which are algebraic \cite{AS}.
Alternatively one could hand-pick the 48 lines on $S$,
such as
\[
x_0+\sqrt[4]{-1}x_1=x_2+\sqrt[4]{-1}x_3=0,
\]
and verify that their Gram matrix attains the maximum rank of 20.
\end{Example}

We can construct a number of further singular K3 surfaces by applying symplectic automorphisms to $S$
and quotienting as in \S\ref{s:Inose}.
Presently, the alternating group $A_4$ acts symplectically by coordinate permutations,
and we can also combine scalings by 4th roots of unity for symplectic automorphisms.

\subsection{Torelli theorem for singular K3 surfaces}

The Torelli theorem states that K3 surfaces are essentially determined by the Hodge structure underlying the transcendental lattice
\begin{eqnarray}
\label{eq:T}
T(X) = \NS(X)^\bot\subset H^2(X,\Z).
\end{eqnarray}
More precisely, given two K3 surfaces $X, X'$,
any effective Hodge isometry of
\[
H^2(X,\Z) \cong H^2(X',\Z)
\]
is induced from a unique isomorphism $X\cong X'$ (cf.~\cite[VIII.11]{BHPV}).
For singular K3 surfaces this can be made very explicit as follows.
In this situation $T(X)$ is a positive definite even lattice of rank 2
which comes with an orientation induced from the regular 2-form.
We can thus identify $T(X)$ with a quadratic form
\begin{eqnarray}
\label{eq:Q}
Q(X) = \begin{pmatrix}
2a & b\\
b & 2c
\end{pmatrix}
\end{eqnarray}
with integer entries $a,c>0$ and discriminant $d=b^2-4ac>0$.
The Torelli theorem can now be formulated as follows:

\begin{Theorem}
\label{thm:Torelli}
Two singular K3 surfaces $X, X'$ are isomorphic
if and only if there is an isometry $T(X)\cong T(X')$.
Equivalently the quadratic forms\linebreak $Q(X), Q(X')$ are conjugate under $\mbox{SL}(2,\Z)$.
\end{Theorem}

\begin{Example}[Fermat quartic cont'd]
\label{ex:Fermat2}
The Fermat quartic $S$
has transcendental lattice represented by the quadratic form
\[
Q(S) =
\begin{pmatrix}
8 & 0\\0 & 8
\end{pmatrix}.
\]
Proving this is a non-trivial task, and the first proper proof seems to go back to Mizukami
\cite{Mizukami}.
(The proof in \cite{PSS} relied on a claim by Demjanenko
which was only later justified by Cassels in \cite{Cassels}.)
Generally this question can be reduced to the problem
whether the lines generate $\NS(S)$ fully or only up to finite index 
(which can be solved using specialisation again, cf.~\cite{SSvL}).
\end{Example}

\subsection{Surjectivity of the period map}

In the moduli context it remains to discuss the surjectivity of the period map.
With the formulation of the Torelli theorem 
at hand,
this amounts to the question
whether all positive definite quadratic forms as in \eqref{eq:Q}
are attained by singular K3 surfaces.
Here it is crucial to note that Kummer surfaces will not be sufficient
since the quotient has the quadratic forms of the abelian surface multiplied by $2$:
\begin{eqnarray}
\label{eq:2Q}
T(\Km(A)) = T(A)[2], \;\;\; \text{i.e.} \;\;\; Q(\Km(A)) = 2Q(A).
\end{eqnarray}
Here transcendental lattice and quadratic form of the abelian surface $A$
are defined in complete analogy with \eqref{eq:T}, \eqref{eq:Q},
and the above relation is valid for complex abelian surfaces of any Picard number.
In essence, we find that Kummer surfaces have 2-divisible transcendental lattices;
this prevents most quadratic forms as in \eqref{eq:Q} to be associated to singular K3 surfaces which are Kummer.

Despite this failure of Kummer surfaces to lead directly to the surjectivity of the period map
for singular K3 surfaces,
they still prove extremely useful.
The first complete argument goes back to Shioda and Inose \cite{SI} 
who proceed in the following two steps:
\begin{enumerate}
\item
refer to work of Shioda and Mitani \cite{SM}
where the corresponding Torelli theorem
including the surjectivity of the period map
is proven for singular abelian surfaces;
\item
prove that any singular K3 surface admits a rational map of degree 2 to a Kummer surface
such that the transcendental lattices differ by the same factor of 2 as in \eqref{eq:2Q}.
\end{enumerate}

Following Morrison \cite{Mo}, the  construction forming the second step 
is nowadays often called a \emph{Shioda-Inose structure}.
We will review it in detail in the next section.
Meanwhile we conclude this section with a brief discussion of the first step.

\subsection{Singular abelian surfaces}
\label{ss:ab}

In analogy with singular K3 surfaces,
a complex abelian surface $A$ is called singular
if its Picard number attains the maximum $\rho(A)=4$.
In \cite{SM} Shioda and Mitani worked out the Torelli theorem
including the surjectivity of the period map.
We summarize their result in the following theorem:

\begin{Theorem}[Shioda--Mitani]
\label{thm:SM}
Isomorphism classes of singular abelian surfaces are in bijective correspondence
with conjugacy classes of binary even positive definite quadratic forms.
\end{Theorem}

The heart of the proof is an explicit construction of singular abelian surfaces
with given transcendental lattice.
In terms of the quadratic form $Q$ as in \eqref{eq:Q},
one exhibits two elliptic curves as complex tori
\begin{eqnarray}\label{eq:E}
\;\;\;\;\;\;\;\;\; E=\C/(\Z+\tau\Z),\;\;\tau = \dfrac{-b+\sqrt{d}}{2a},\;\;\;\;\; E'=\C/(\Z+\tau'\Z),\;\;\tau' = \dfrac{b+\sqrt{d}}2.
\end{eqnarray}
Then Shioda and Mitani compute that the product $E\times E'$
has transcendental lattice exactly represented by $Q$.

\begin{Remark}
\label{rem:CM}
The terminology "singular abelian/K3 surface" is indeed appropriate in the following sense:
the elliptic curves $E, E'$ in \eqref{eq:E} have complex multiplication (CM) (in fact, they are isogenous);
classically their j-invariants are called "singular".
\end{Remark}

\section{Shioda-Inose structures}
\label{s:SI}

Given a quadratic form $Q$ as in \eqref{eq:Q},
\ref{ss:ab} exhibits elliptic curves $E, E'$ such that $T(\Km(E\times E'))$
is represented by $2Q$.
In order to recover the original quadratic form
on a singular K3 surface,
Shioda and Inose developed a geometric construction that applies generally to 
Kummer surfaces of product type.
An instrumental ingredient consists in the so-called double Kummer pencil
formed by 24 rational curves on $\Km(E\times E')$.
In terms of the elliptic fibrations over $\PP^1$ induced by the projections onto either factor,
the curves constitute the four singular fibers (type $I_0^*$, 5 components each)
and the four 2-torsion sections.
We sketch the curves in the following figure
where the fibrations could be thought of both in the horizontal or vertical direction:

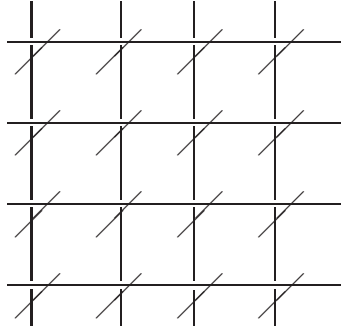
\begin{figure}[ht!]
\begin{center}
\setlength{\unitlength}{1.08mm}
\begin{picture}(60, 47)
%
%

{
\put(7, 10){\line(1, 0){41}}}

{
\put(7, 20){\line(1, 0){41}}
\put(7, 30){\line(1, 0){41}}}
\put(7, 40){\line(1, 0){41}}
%
%
{
\put(10, 5.1){\line(0, 1){4.4}}
\put(10, 10.6){\line(0, 1){8.9}}
\put(10, 20.6){\line(0, 1){8.9}}
\put(10, 30.6){\line(0, 1){8.9}}
\put(10, 40.6){\line(0, 1){4.4}}
\put(21, 5.1){\line(0, 1){4.4}}
\put(21, 10.6){\line(0, 1){8.9}}
\put(21, 20.6){\line(0, 1){8.9}}
\put(21, 30.6){\line(0, 1){8.9}}
\put(21, 40.6){\line(0, 1){4.4}}}

{
\put(30, 5){\line(0, 1){4.4}}
\put(30, 10.6){\line(0, 1){8.9}}
\put(30, 20.6){\line(0, 1){8.9}}
\put(30, 30.6){\line(0, 1){8.9}}
\put(30, 40.6){\line(0, 1){4.4}}
\put(40, 5.1){\line(0, 1){4.4}}
\put(40, 10.6){\line(0, 1){8.9}}
\put(40, 20.6){\line(0, 1){8.9}}
\put(40, 30.6){\line(0, 1){8.9}}
\put(40, 40.6){\line(0, 1){4.4}}}
%
%

\put(8, 6){\line(1,1){5.5}}
{
\put(8, 16){\line(1,1){5.5}}
\put(8, 26){\line(1,1){5.5}}
\put(8, 36){\line(1,1){5.5}}}
%
{
\put(18, 6){\line(1,1){5.5}}
\put(18, 16){\line(1,1){5.5}}}
\put(18, 26){\line(1,1){5.5}}
\put(18, 36){\line(1,1){5.5}}
%

{
\put(28, 6){\line(1,1){5.5}}}
\put(28, 16){\line(1,1){5.5}}
\put(28, 26){\line(1,1){5.5}}
{
\put(28, 36){\line(1,1){5.5}}
%

\put(38, 6){\line(1,1){5.5}}}
\put(38, 16){\line(1,1){5.5}}
\put(38, 26){\line(1,1){5.5}}

{
\put(38, 36){\line(1,1){5.5}}}
%
\end{picture}
\end{center}
\vskip -.3cm
\caption{Double Kummer pencil}\label{fig:SI}
\end{figure}

The key property for the considerations of \cite{SI} is
that Kummer surfaces of product type admit several distinct elliptic fibrations.
In the generic situation these were later classified by Oguiso \cite{O};
Shioda and Inose exploit this feature by exhibiting a specific fibration on $\Km(E\times E')$
by singling out a divisor $D$ of Kodaira type $II^*$ in the double Kummer pencil.
Recall from \S\ref{s:ell} that the divisor $D$ (printed in blue in Figure \ref{fig:SI2})
induces indeed an elliptic fibration 
\[
\Km(E\times E') \to \PP^1.
\]
Moreover the rational curve $C$ in green meets $D$ transversally at exactly one point, 
so $C$ will serve as a section of the new fibration.

\begin{figure}[ht!]
\begin{center}
\setlength{\unitlength}{1.08mm}
\begin{picture}(60, 47)
%
%

{\color{green}
\put(7, 10){\line(1, 0){41}}}

{\color{blue}
\put(7, 20){\line(1, 0){41}}
\put(7, 30){\line(1, 0){41}}}
\put(7, 40){\line(1, 0){41}}
%
%
{\color{blue}
\put(10, 5.1){\line(0, 1){4.4}}
\put(10, 10.6){\line(0, 1){8.9}}
\put(10, 20.6){\line(0, 1){8.9}}
\put(10, 30.6){\line(0, 1){8.9}}
\put(10, 40.6){\line(0, 1){4.4}}
\put(21, 5.1){\line(0, 1){4.4}}
\put(21, 10.6){\line(0, 1){8.9}}
\put(21, 20.6){\line(0, 1){8.9}}
\put(21, 30.6){\line(0, 1){8.9}}
\put(21, 40.6){\line(0, 1){4.4}}}

{\color{red}
\put(30, 5){\line(0, 1){4.4}}
\put(30, 10.6){\line(0, 1){8.9}}
\put(30, 20.6){\line(0, 1){8.9}}
\put(30, 30.6){\line(0, 1){8.9}}
\put(30, 40.6){\line(0, 1){4.4}}
\put(40, 5.1){\line(0, 1){4.4}}
\put(40, 10.6){\line(0, 1){8.9}}
\put(40, 20.6){\line(0, 1){8.9}}
\put(40, 30.6){\line(0, 1){8.9}}
\put(40, 40.6){\line(0, 1){4.4}}}
%
%

\put(8, 6){\line(1,1){5.5}}
{\color{blue}
\put(8, 16){\line(1,1){5.5}}
\put(8, 26){\line(1,1){5.5}}
\put(8, 36){\line(1,1){5.5}}}
%
{\color{blue}
\put(18, 6){\line(1,1){5.5}}
\put(18, 16){\line(1,1){5.5}}}
\put(18, 26){\line(1,1){5.5}}
\put(18, 36){\line(1,1){5.5}}
%

{\color{red}
\put(28, 6){\line(1,1){5.5}}}
\put(28, 16){\line(1,1){5.5}}
\put(28, 26){\line(1,1){5.5}}
{\color{red}

\put(28, 36){\line(1,1){5.5}}
%

\put(38, 6){\line(1,1){5.5}}}
\put(38, 16){\line(1,1){5.5}}
\put(38, 26){\line(1,1){5.5}}

{\color{red}
\put(38, 36){\line(1,1){5.5}}}
%
\end{picture}
\end{center}
\vskip -.3cm
\caption{Divisor of Kodaira type $II^*$ in the double Kummer pencil}\label{fig:SI2}
\end{figure}
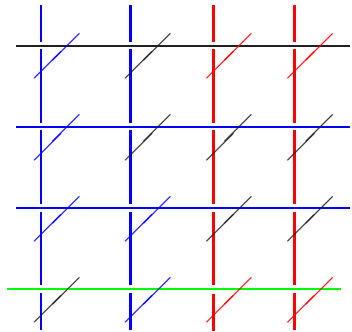

There are 6 rational curves in the double Kummer pencil 
(printed in red in Figure \ref{fig:SI2})
which do not meet the divisor $D$ of Kodaira type $II^*$.
Hence they are components of other fibers.
In fact, since both perpendicular 3-chains meet the section $C$,
they are located in different fibers.
A detailed analysis shows that the fiber types can only be $I_0^*, I_1^*$ or $IV^*$.
In fact, unless $E\cong E'$, both singular fibers have type $I_0^*$.

Shioda and Inose proceed by a quadratic base change
that ramifies exactly at the above two singular fibers. 
Pull-back from $\Km(E\times E')$ gives another jacobian elliptic surface $X$
with the $II^*$ fiber duplicated while the two fibers in the branch locus
are generically replaced by smooth fibers. 
The Euler number $e(X)=24$ reveals that $X$ is again K3;
in fact, since $X$ dominates $\Km(E\times E')$ by construction,
we have $\rho(X)=\rho(\Km(E\times E'))$ by Proposition \ref{Prop:Inose}.
One concludes by verifying by an explicit calculation that
\[
T(X) \cong T(\Km(E\times E'))[1/2].
\]
Together with \eqref{eq:2Q}, this gives
\[
T(X) \cong T(A),
\]
and the surjectivity of the period map for singular K3 surfaces
follows from Theorem \ref{thm:SM}. \qed

\medskip

To conclude this section, let us summarize the above construction:
a K3 surface $X$ with a rational map of degree 2 to a Kummer surface
which induces multiplication by 2 on the transcendental lattices:

 \begin{eqnarray*}
 \label{eq:SI}
  \xymatrix{A \ar@{-->}[dr] && X\ar@{-->}[dl]&\\
 & \Km(A)&& T(A)\cong T(X)}
 \end{eqnarray*}

In generality, Morrison coined the terminology \emph{Shioda-Inose structure}
for such a diagram. 
In \cite{Mo} he proved that any K3 surface with $\rho\geq 19$ admits a Shioda-Inose structure,
and he worked out explicit criteria for K3 surfaces with $\rho=17,18$.
Shioda-Inose structures are a versatile tool for the study of K3 surfaces;
they turned up during the workshop
especially in the context of classification and moduli problems 
(see also \cite{S-SI} and the references therein).

Before continuing our investigation of singular K3 surfaces, eventually aiming for zeta functions
and modularity,
we take a little detour towards Mordell-Weil ranks of elliptic K3 surfaces.

\section{Mordell-Weil ranks of elliptic K3 surfaces}
\label{s:MW}

By the Shioda-Tate formula \eqref{eq:ST}, 
a jacobian elliptic K3 surface can have Mordell-Weil rank at most $18$ over $\C$.
Recall from \S\ref{s:ell} how elliptic fibrations on K3 surfaces are governed by lattice theory.
In order to see which MW-ranks are attained, 
one can therefore appeal to the moduli theory of lattice polarised K3 surfaces,
combined with the theory of Mordell-Weil lattices \cite{ShMW}.
The solution was first given by Cox \cite{Cox}:

\begin{Theorem}[Cox]
Any integer between 0 and 18 occurs as Mordell-Weil rank of a complex elliptic K3 surface.
\end{Theorem}

In the spirit of \S\ref{s:rho=1},
it is still a delicate problem to exhibit an elliptic K3 surface with given MW-rank.
Kuwata gave an almost complete answer in \cite{Kuwata}
based on the Shioda-Inose structure related to $E\times E'$.
Here we will briefly review his construction.

Recall from the previous section 
that the Kummer surface $\Km(E\times E')$ is covered by another K3 surface $X$
which is equipped with an induced jacobian elliptic fibration with 2 fibers of Kodaira type $II^*$.
Much like what we did to the original elliptic fibration on $\Km(E\times E')$
one can try to apply base changes to $X$ which remain K3.
Here's the first example:

\begin{Example}
\label{ex:(2)}
A quadratic base change ramified at the 2 special fibers
gives another K3 surface $X^{(2)}$ (by Euler number considerations).
The induced elliptic fibration has two fibers of type $IV^*$ instead of $II^*$ (see the small table below).
It is a non-trivial fact that $X^{(2)}$ indeed returns the Kummer surface $\Km(E\times E')$
(which thus sandwiches $X$, see \cite{Sandwich}).
The elliptic fibration can be understood in terms of \eqref{eq:Km2} as projection onto $\PP^1_y$
(i.e.~a quadratic base change of a cubic pencil -- the pencil is the corresponding quadratic twist of $X$).
\end{Example}

In general, a fiber of type $II^*$ (outside characteristics $2,3$)
pulls back as follows under a cyclic base change of degree $n$:
$$
\begin{array}{c|cccccc}
n\mod 6 & 0 & 1 & 2 & 3 & 4 & 5\\
\hline
\text{fiber type} & 
I_0 & II^* & IV^* & I_0^* & IV & II
\end{array}
$$
Hence a simple Euler number computation allows us to determine all base changes of $X$ 
which remain K3.
This requires that the base change is cyclic of degree $n\leq 6$ and 
exactly ramifies at the 2 special fibers of type $II^*$ (as in Example \ref{ex:(2)}).
By construction, the resulting K3 surface $X^{(n)}$ automatically comes with a rational map of degree $n$ to $X$;
hence by Proposition \ref{Prop:Inose}
\begin{eqnarray}
\label{eq:n}
\rho(X^{(n)}) = \rho(X) = \rho(\Km(E\times E')).
\end{eqnarray}
In consequence, the MW-rank of $X^{(n)}$ can be computed depending only on $E$ and $E'$.
Essentially this relies on $E$ and $E'$ being isogenous or not:
\begin{eqnarray*}
\label{eq:rho-p}
\rho(\Km(E\times E'))=18+\mbox{rank}(\mbox{Hom}(E,E')) =
\begin{cases}
 18, & \text{if $E\not\sim E'$;}\\
 19, & \text{if $E\sim E'$ without CM};\\
 20, & \text{if $E\sim E'$ with CM}.
              \end{cases}
\end{eqnarray*}
In order to apply the Shioda-Tate formula \eqref{eq:ST},
 one further needs that the fibrations have additional reducible fibers
if and only if $E$ and $E'$ are isomophic.
With the exception of the cases $j(E),j(E')\in\{0,1728\}$,
the results are summarised in Table \ref{Table:Kuwata}.

\begin{table}[ht!]
$$
\begin{array}{c||cc||c|c|c}
\hline
& \multicolumn{2}{c||}{E\cong E'} 
& E\not\cong E' & E\sim E' & E\not\sim E'\\
n & \text{config} & \mbox{MW-rank} & \text{config} & \mbox{MW-rank} & \mbox{MW-rank}\\
\hline
1 & 2\,II^*, I_2, 2\,I_1 & \begin{cases} 1\\0\end{cases} & 2\,II^*, 4\,I_1 & \begin{cases} 2\\1\end{cases} & 0\\
2 & 2\,IV^*, 2\,I_2, 4\, I_1 & \begin{cases} 4\\3\end{cases} & 2\,IV^*, 8\,I_1 & \begin{cases} 6\\5\end{cases} & 4\\
3 & 2\,I_0^*, 3\,I_2, 6\,I_1 &\begin{cases} 7\\6\end{cases} & 2\,I_0^*, 12\, I_1 & \begin{cases} 10\\9\end{cases} & 8\\
4 & 2\,IV, 4\,I_2,8\,I_1 & \begin{cases} 10\\9\end{cases} & 2\,IV, 16\,I_1 & \begin{cases} 14\\13\end{cases} & 12\\
5 & 2\,II, 5\,I_2, 10\,I_1 & \begin{cases} 13\\12\end{cases} & 2\,II, 20\,I_1 & \begin{cases} 18\\17\end{cases} & 16\\
6 &  6\,I_2, 12\,I_1 & \begin{cases} 12\\11\end{cases} & 24\,I_1 & \begin{cases} 18\\17\end{cases} & 16\\
\hline
\end{array}
$$
\caption{Fiber configuration and Mordell-Weil rank of $X^{(n)}$}
\label{Table:Kuwata}
\end{table}

The only MW-rank missing from Table \ref{Table:Kuwata} is 15.
This gap was subsequently filled by Kloosterman.
In \cite{Kl-15}, he exhibited an explicit elliptic K3 surface with MW-rank 15
much along the lines of \S\ref{s:K3-1}.
In brief he worked out a 3-dimensional family of elliptic K3 surfaces
with generic MW-rank 15 (again using base change from an appropriate elliptic K3 surface).
Then the specialisation technique from \ref{ss:vL}, \ref{ss:Kl} enabled him to single out a general member
of the family.

\begin{Remark}
Extending \eqref{eq:n},
Shioda proved in \cite{Sh-SP} that $T(X^{(n)})=T(X)[n]$.
\end{Remark}

One can enrich the arithmetic flavour of the MW-rank problem
by specifying the ground field in consideration.
Naturally one might wonder which MW-ranks are attained by elliptic K3 surfaces over $\Q$
(i.e.~as elliptic curves over $\Q(t)$).
The answer is due to Elkies (see \cite{Elkies-rank}):

\begin{Theorem}[Elkies]
\label{thm:17}
The maximal MW-rank of an elliptic K3 surface over $\Q$ is 17.
\end{Theorem}

On the existence side of Theorem \ref{thm:17},
Elkies exhibits an explicit K3 surface of MW-rank 17 over $\Q$.
The construction uses families of K3 surfaces and ingenious specialisation and lifting arguments
(somewhat comparable to the concepts in \S\ref{s:mod}).
As an application, Elkies is able to specialise to elliptic curves over $\Q$ of even higher rank,
pushing the previous record ranks as far as 28.

In order to explain the non-constructive part of Elkies' proof, 
we have to discuss fields of definition of singular K3 surfaces first.

\section{Fields of definition of singular K3 surfaces}
\label{s:fields}

Since singular K3 surfaces lie isolated in the moduli space,
they are always defined over some number field.
In this section, we will be more specific about the number field
and also comment on known obstructions.

So far, we have not given explicit equations for singular K3 surfaces
except for the Fermat quartic (Ex.~\ref{ex:Fermat}).
To remedy this,
consider a singular K3 surface $X$ with elliptic fibration coming from the Shioda-Inose structure.
By Tate's algorithm \cite{Tate},
the two fibers of type $II^*$ determine the Weierstrass equation of $X$ almost completely.
The remaining coefficients depend on Weber functions in terms of the j-invariants of the elliptic curves,
as determined by Inose \cite{Inose}:
\begin{eqnarray}
\label{eq:Inose}
X:\;\;\; y^2 = x^3 - 3 A t^4 x + t^5 (t^2 - 2  B t+1)
\end{eqnarray}
where $ A^3 = j(E)\,j(E')/12^6, B^2 = (1-j(E)/12^3)\,(1-j(E')/12^3).$
An easy twist reveals that the above fibration in fact admits a model over $\Q(j(E), j(E'))$ (see \cite[Prop.~4.1]{S-fields}).
This field is related to class field theory as follows.
Let $d$ denote the discriminant of $X$,
i.e.~$d=-\det Q(X)$ where $Q(X)$ is the quadratic form representing $T(X)$.
If $K=\Q(\sqrt{d})$, then class group theory attaches to $d$ a ring class field $H(d)$
as  maximal abelian extension of $K$ with prescribed ramification.
The Galois group of the extension $H(d)/K$ is canonically isomorphic to the class group $Cl(d)$,
consisting of primitive quadratic forms of discriminant $d$ as in \eqref{eq:Q} with Gauss composition.
CM theory of elliptic curves states that
\[
H(d) = K(j(E')) = K(j(E),j(E')).
\]
In summary we find
\begin{Proposition}
\label{prop:H(d)}
Any singular K3 surface of discriminant $d$ admits a model over the ring class field $H(d)$.
\end{Proposition}
As satisfying as the above canonical field of definition may be,
in practice it is often far from optimal.
As an illustration we return to the Fermat quartic once again.

\begin{Example}
\label{ex:descent}
In Example \ref{ex:Fermat2} we have given the transcendental lattice of the Fermat quartic $S$.
In the realm of Shioda-Inose structures,
$S$ thus arises from the elliptic curves with periods $\tau=\sqrt{-1}, \tau'=4\sqrt{-1}$.
Note the discrepancy that the latter CM curve is only defined over $\Q(\sqrt{2})$
(and so is Inose's model \eqref{eq:Inose})
while the Fermat quartic has the obvious model over $\Q$.
In fact, considering the Fermat quartic as a Kummer surface instead,
it is associated to the elliptic curves with periods $\tau=\sqrt{-1}, \tau'=2\sqrt{-1}$
which are indeed both defined over $\Q$.
\end{Example}

Example \ref{ex:descent} leads to the problem of working out 
obstructions for singular K3 surfaces to descend from the ring class field $H(d)$ to smaller fields,
in particular to $\Q$.
Essentially there are two obstructions known,
one coming from the transcendental lattice (see \cite{S-fields}),
the other imposed by the N\'eron-Severi lattice.
Here we shall only consider the latter obstruction
which was discovered independently by Elkies and the author (cf.~\cite{S-NS}).
In short, it states that the ring class field $H(d)$ is essentially preserved through the Galois action on 
the N\'eron-Severi lattice:
\begin{Theorem}[Elkies, Sch\"utt]
\label{thm:NS}
Let $X$ be a singular K3 surface of discriminant $d$.
If $\NS(X)$ has generators defined over some number field $L$,
then 
\[
L(\sqrt d)\supset H(d).
\]
\end{Theorem}
The theorem paves the way towards finiteness classifications.
For instance, one readily deduces that there are only finitely many singular K3 surfaces over $\Q$
(up to $\bar\Q$-isomorphism)
since the N\'eron-Severi lattice cannot admit Galois actions by arbitrarily large groups.
To see this, note that the rank of the lattice is trivially constant and 
the hyperplane class is always preserved by Galois;
this leaves a faithful Galois action on the negative-definite complement.

\subsection{Mordell-Weil ranks over $\Q$}

Returning to the problem of Mordell-Weil ranks over $\Q$,
we note that rank 18 would imply that all of $\NS$ would be defined over $\Q$.
By Theorem \ref{thm:NS}, the discriminant $d$ could only have class number one, i.e.~
\[
d=-3,-4,-7,-8,-11,-12,-16,-19,-27,-28,-43,-67,-163.
\]
Elkies continues to argue with the Mordell-Weil lattice $M$ of the given elliptic fibration.
By \cite{ShMW}, this is an even positive-definite lattice 
of rank 18 and discriminant $d$ without roots (i.e.~$x^2>2$ for all $x\in M$).
Note that with discriminant $|d|\leq 163$
such a lattice would break the known density records for sphere packings,
but this can only be regarded as evidence against Mordell-Weil rank 18 over $\Q$.

In order to rule out the existence of such a $M$,
one appeals to general results of lattice theory developed by Nikulin \cite{N}.
Geared towards  elliptic K3 surfaces,
they imply  that  $M$ admits a primitive embedding into some Niemeier lattice
(i.e.~one of the 24 unimodular positive-definite even lattices of rank 24, cf.~\cite{Nie}).
Its orthogonal complement $L$ (positive-definite of rank 6)
is called the partner lattice of $M$.
Often the partner lattice $L$ can be determined a priori from $\NS$ or from the transcendental lattice $T(X)$.
In fact, all jacobian elliptic fibrations on a given complex K3 surface $X$ can be classified
in terms of the primitive embeddings of $L$ into Niemeier lattices.
Full classifications have been established, for instance, by Nishiyama \cite{Nishi}
for the singular K3 surfaces with discriminant $-3$ and $-4$.
Here there are 6 resp.~13 jacobian fibrations, but
the Mordell-Weil rank does never exceed 1.

Returning to the problem of  Mordell-Weil rank 18,
ruling this out for a single K3 surface
amounts to proving that
the partner lattice $L$ does not admit a primitive embedding into any Niemeier lattice
without any perpendicular roots.
For the 13 discriminants of class number one,
this is exactly what Elkies succeeds in doing,
thus ruling out Mordell-Weil rank 18 over $\Q$.
In fact, even Mordell-Weil rank 17 is only barely attained over $\Q$
by an elliptic K3 surface of discriminant $12\cdot 79$ (see \cite{Elkies-rank}).

\section{Modularity of singular K3 surfaces}
\label{s:mod}

The modularity of elliptic curves over $\Q$ has been one of the biggest achievements in mathematics in
the last 20 years (implying in particular Fermat's Last Theorem).
Starting from an elliptic curve $E$ over $\Q$, one considers the reductions $E_p$ for all good primes.
Counting points on $E_p$ over $\F_p$, one defines the quantity
\begin{eqnarray}
\label{eq:Eap}
a_p = 1+p-\# E_p(\F_p).
\end{eqnarray}
The Taniyama-Shimura-Weil conjecture states that the collection of integers $\{a_p\}$ comprise the Fourier coefficients
of a single modular form (the eigenvalues of a Hecke eigenform of weight $2$).
This is now a theorem, thanks to the work of Wiles \cite{Wiles} and others.
Historically the first case to be settled comprised elliptic curves with CM.
Generally over any number field,
their zeta function was shown by Deuring \cite{Deuring} to depend on certain Hecke characters $\psi$.
For elliptic curves over $\Q$ with CM, this has an incarnation in terms of modular forms with CM.

For singular K3 surfaces, the Shioda-Inose structure allows one 
to express the zeta function in terms of the Hecke character $\psi^2$,
but a priori only over some extension of $H(d)$ required for exhibiting the construction \cite[Thm.~6]{SI}.
Especially in view of the possible descent from $H(d)$ to some smaller fields as explored in \S\ref{s:fields},
it is thus an independent task to prove modularity for singular K3 surfaces over $\Q$
(long predicted by standard conjectures for 2-dimensional Galois representations).
Here \eqref{eq:Eap} is rephrased by virtue of Lefschetz' fixed point formula.
For a singular K3 surface $X$ over $\Q$ and its good reductions $X_p$, we define
\begin{eqnarray}
\label{eq:Xbp}
\#X_p(\F_p) = 1 + b_p + hp + p^2
\end{eqnarray}
where the integer $h$ encodes the action of Frob$_p^*$ on $\NS(X_p\otimes\bar\F_p)$.
In this context, the integers $b_p$ ought to belong to a Hecke eigenform of weight $3$.
This was solved by Livn\'e in a more general framework of orthogonal Galois representations \cite{L}:

\begin{Theorem}[Livn\'e]
\label{thm:mod}
Every singular K3 surface over $\Q$ is modular.
Its zeta function is expressed in terms of a Hecke eigenform of weight 3
with CM by $\Q(d)$ where $d$ denotes the discriminant of the K3 surface.
\end{Theorem}

\begin{Example}
The zeta function of the Fermat quartic (the model $S$ over $\Q$ from Ex.~\ref{ex:Fermat})
can be expressed in terms of the eta product $\eta(4\tau)^6$ 
(the weight 3 eigenform of level 16 from \cite[Table 1]{S-CM}).
With Legendre symbols $\chi_\bullet$ accounting for the Galois action on the lines defined over $\Q(\sqrt{-1},\sqrt 2)$,
one finds
\[
\zeta(S,s)=\zeta(s) \zeta(s-1)^5 \zeta(\chi_{-1},s-1)^3 \zeta(\chi_2,s-1)^6 \zeta(\chi_{-2},s-1)^6 L(\eta(4\tau)^6,s) \zeta(s-2).
\]
Spelling this out for $\#S_p(\F_p)$ as in \eqref{eq:Xbp},
we find $h=5+3\chi_{-1}(p) + 6(\chi_2(p)+\chi_{-2}(p))$ and Fourier coefficients $b_p$ determined by the infinite product
\[
\eta(4\tau)^6 = q \prod_{n\geq 1} (1-q^{4n})^6 = \sum_{n\geq 1} b_nq^n.
\]
\end{Example}

We emphasize how the overall picture differs from the case of elliptic curves over $\Q$.
There Shimura had shown in the 50's how to associate 
an elliptic curve over $\Q$ as a factor of the Jacobian $J_0(N)$
to a Hecke eigenform of weight 2 and level $N$ with eigenvalues in $\Z$,
but the modularity remained open for more than 30 years.
For singular K3 surfaces, quite on the contrary, modularity was proved first, but
only a few years ago it became clear that 
they also sufficed in the opposite direction:

\begin{Theorem}[Elkies-Sch\"utt {\cite{ES}}]
\label{thm:ES}
Every known Hecke eigenform of weight 3 with eigenvalues in $\Z$
is associated to a singular K3 surface over $\Q$.
\end{Theorem}

In view of Theorem \ref{thm:mod} it is crucial to note that the CM property is by no means special,
but rather forced in odd weight by real Hecke eigenvalues due to a result of Ribet \cite{Ribet}.
Then the key step towards Theorem \ref{thm:ES} consists in a classification of CM form with eigenvalues in $\Z$
which narrows the problem down to an essentially finite problem \cite{S-CM}.
In practice it therefore suffices to exhibit a singular K3 surface over $\Q$ for any imaginary quadratic field
whose class group has exponent 1 or 2.
There are 65 such fields known, with at most one further field possible by \cite{Wb}.
Thus the restriction of Theorem \ref{thm:ES} which, for instance, becomes vacuous 
if one is willing to assume the extended Riemann hypothesis for odd real Dirichlet characters.

We conclude this paper with a few words towards to construction of these singular K3 surfaces over $\Q$ in \cite{ES}.
After exhausting the examples occurring in the literature,
we started considering 1-dimensional families of K3 surfaces with $\rho\geq 19$.
These have dense specialisations with $\rho=20$ over $\bar\Q$,
parametrised by some modular or Shimura curve,
but the determination of the CM points is a non-trivial task.
For instance, one can deform the Fermat quartic to the so-called Dwork pencil
\begin{eqnarray}
\label{eq:Dwork}
S_\lambda = \{x_0^4+x_1^4+x_2^4+x_3^4=\lambda x_0x_1x_2x_3\}\subset\PP^3.
\end{eqnarray}
This retains a big part of the automorphism group which partly explains why $\rho(S_\lambda)\geq 19$.
We will briefly come back to this family below;
for a detailed account, the reader is referred to the upcoming paper \cite{ES-family}.

A common theme among the constructions is the use of jacobian elliptic fibration (again!),
preferably with large contribution from the singular fibers 
(such that ideally the MW-rank would be zero generically).
This has both advantages on the constructive side of exhibiting the families
and for the decisive step of provably determining CM points (see below).
Note that by Theorem \ref{thm:NS},
a singular K3 surface of big class number over $\Q$ has to admit a fairly big Galois action on $\NS$,
so the families necessarily become pretty complicated (despite the low MW-rank).

Finally we comment on the possible approaches to determine the CM points of a 
given family such as the Dwork pencil \eqref{eq:Dwork}.
One possibility to proceed would be to determine the moduli structure of the parametrising curve
and compute the CM points explicitly. 
For instance, the parameter of the Dwork pencil can be interpreted as 4-fold cover
of the modular curve $X^*(2)$ parametrising pairs of 2-isogenous elliptic curves 
(where the CM-points can be calculated easily).
Geometrically this can be achieved through a Shioda-Inose structure over $\Q$.
However, with the families getting more and more complicated, the determination of the moduli curve soon becomes infeasible (let alone the calculation of CM points),
and in fact these 1-dimensional families of K3 surfaces have proved a versatile tool
to work out explicit models of Shimura curves (see \cite{E-Shimura}).
Here is how we proceeded instead in \cite{ES}.

Experimentally one can count points over finite fields 
throughout the family and apply  Lefschetz' fixed point formula \eqref{eq:Xbp}.
Fixing a target modular form, one can sieve those parameters modulo several primes $p$
which would fit the Fourier coefficients $b_p$ of the modular form.
Once a collection of residue parameters is computed, one can try to lift them to a single parameter
of small height in $\Q$.
This procedure is surprisingly effective (though not for the Dwork pencil in the above form,
because it admits an order 4 symmetry in the parameter $\lambda$).
It remains to prove the CM points that one might have computed experimentally.

Proving the CM points amounts to exhibiting an extra divisor on the special member of the family.
Here
one takes advantage of the structure as jacobian elliptic surface,
together with an extra bit of information extracted from Theorem \ref{thm:mod}.
Namely the precise modular form that we are aiming at
predicts the square class of the discriminant of the K3 surface $X$
that we want to prove to have $\rho=20$.
Presently this usually requires the presence of an additional section.
Thanks to the height pairing from the theory of Mordell-Weil lattices \cite{ShMW},
information about the conjectural discriminant can be translated into details
which fiber components have to be met by the section.
If the height of the section is small enough,
this gives enough information about the shape of the section 
to solve for it either directly or with $p$-adic methods
(further simplified by the fact that we have already a candidate for the parameter of the family).

In the end, it turns out that the two obstructions -- the families being rich enough
and the height of the section being small enough to allow for explicit calculations --
balance out just right to compute singular K3 surfaces for all but a handful of the known Hecke eigenforms
(which require special treatment).

\medskip

We end the paper by remarking
that the analogous problem in higher dimension,
realising all Hecke eigenforms of fixed weight with eigenvalues in $\Z$ by a single class of varieties such as Calabi-Yau manifolds, seems wide open even in dimension 3 where some modularity results are available
(see Yui's introductory lectures).

\subsection*{Acknowledgements}

It is a great pleasure to thank the other organisers of the Fields workshop
and particularly all the participants
for creating such a stimulating atmosphere.
Special thanks to the Fields Institute for the great hospitality
and to the referee for his comments.
These lecture notes were written down while the author enjoyed
support from the ERC under StG 279723 (SURFARI)
which is gratefully acknowledged.

\end{document}